\documentclass{svproc}
\usepackage{url}

\usepackage{amsmath}
\usepackage{amsfonts}
\usepackage{amssymb}
\usepackage{helvet}
\usepackage{lipsum}
\usepackage{textpos}
\allowdisplaybreaks
\nonstopmode \numberwithin{equation}{section}

\begin{document}

\mainmatter              
\title{
Non-instantaneous Impulsive Riemann-Liouville Fractional Differential Systems: Existence and Controllability Analysis}
\titlerunning{Non-instantaneous impulsive fractional order systems}  
%
\author{Lavina Sahijwani\inst{1}, N. Sukavanam\inst{2} and Abdul Haq\inst{3}
}
\authorrunning{Lavina Sahijwani, N. Sukavanam, Abdul Haq} 
%
%
\institute{Department of Mathematics,\\
Indian Institute of Technology Roorkee, Roorkee, India,\\
\email{sahijwani.lavina@gmail.com},\\
\and
\email{nsukvfma@iitr.ac.in}
\and
\email{abdulhaqiitr@gmail.com}
}

\maketitle              

\begin{abstract}

The article is dedicated towards the study of fractional order non-linear differential systems with non-instantaneous impulses involving Riemann-Liouville derivatives with fixed lower limit and appropriate integral type initial conditions in Banach spaces. First, mild solution of the system is constructed and subsequently its existence is proven using Banach's fixed point theorem. Then, results of approximate controllability are established using concept of fractional semigroup and an iterative technique. 
Suitable examples are given in the end supporting the methodology along with pointing out correction in examples presented in previous articles.




\end{abstract}
\section{Introduction} \label{1}
\indent Controllability is the qualitative property of steering any dynamical system from initial arbitrary position to any desired final position utilizing appropriate control functions within stipulated time. Control theory, being a multidisciplinary branch stemmed from mathematics to engineering, has wide-ranging implementation in robotics, aeronautical and automobile engineering, image processing, biomathematical modelling and appreciably more. Control theory, in spaces of infinite and finite dimensions, have thoroughly been discussed in \cite{curtain} and \cite{SB1} respectively.
 The conception of controllability was first initiated and established by Kalman \cite{kalman} in 1960, and since then it is the matter of prime importance for the researchers worldwide. 
The results of existence and controllability
 for various differential systems of integer and fractional order involving Riemann-Liouville and Caputo derivatives have closely been demonstrated in many artefacts (refer \cite{rp ag1,rp ag main,bala,SB1,book hindawi,curtain,de gruyter,abdul,abdul:partial,donal jmaa,topo 2015,hey:podlubny,kalman,kilbas:srivast,visco,surendra:j.diff,muslim,diffusion,liu:siam,mahm2,mahm3,oldh:span,pazy,jde1,podlubny,min zhang,zhou}and references therein).\\ 
\indent The study of fractional calculus has long been admired from past three decades. The first work, exclusively committed 
to the study of fractional calculus, is the book by Oldham and Spanier {\cite{oldh:span}}, 1974.
Fractional derivatives serve as an exemplary 
mechanism for the 
interpretation of heritable
properties and memory of profuse scientific, physical and engineering phenomena. On account of finer accuracy and precision over integer-order models, fractional derivatives accelerate its applications in diffusion process, biological mathematical models, aerodynamics, viscoelasticity, electrical engineering, signal and image processing, control theory, heat equation, electricity mechanics, electrodynamics of complex medium, etc. (see \cite{bala,kilbas:srivast,visco,surendra:j.diff,diffusion,podlubny}).\\
\indent In domain of fractional calculus, Riemann-Liouville and Caputo type derivatives have maintained to be the centre of attention for numerous analysts. Riemann-Liouville derivative shows supremacy over Caputo in the sense that it allows the function involved to bear discontinuity at origin. On the other hand, it  doesn't allow the use of traditional initial conditions, the initial conditions involved in Riemann-Liouville case are either in the integral form or are weighted initial conditions. Heymans and Podlubny \cite{hey:podlubny} were the ones accredited for the manifestation of physical significance to the initial conditions used in 
regard of Riemann-Liouville fractional order viscoelastic systems.\\
\indent Recently, researchers worldwide gravitate towards the analysis of impulsive evolution systems with the focus on suitable mathematical modelling, existence of integral solution, its stability, controllability and much more. In realistic modelling,
the occurrence of perturbations, due to extrinsic intercessions, 
is inevitable, yet, unpredictable. These perturbations or sudden changes are nothing but impulses affecting the solution's behaviour majorly. The literature consists of artefacts
addressing mainly two types of impulses, one is, instantaneous impulses which happen for a very smaller period of time, and the other is, non-instantaneous impulses, which start all at once but hold for a finite time interval, for example, injection of a medical drug into the human body is sudden but takes time to stabilize its effect. The books \cite{book hindawi,de gruyter} and articles \cite{rp ag1,rp ag main,donal jmaa,topo 2015,min zhang} contribute to the study of various impulsive differential systems. \cite{rp ag1,muslim} addresses the fractional differential systems involving Caputo derivative with non instantaneous impulses. 
There is no such article in the literature so far addressing the analysis for approximate controllability of Riemann-Liouville fractional evolution systems having non-instantaneous impulses, and hence, is the motivation for the present artefact.\\ 
\indent The study of this article revolves around the following system:
\begin{eqnarray}
&&_{0}D_{t}^{\eta} z(t) = Az(t) + Bu(t) + h(t,z(t)), ~~~~~ t\in \cup_{r=0}^{m} (p_r, t_{r+1}], \label{system}\\ \nonumber
&&z(t)= \psi_r(t,z(t_{r}^{-})), ~~~ t\in (t_r, p_r],~~~ r = 1,2,....,m,\\ \nonumber
&&_{p_r}I_{t}^{1-\eta}z(t)|_{t=p_r} = \psi_r(p_r,z(t_{r}^{-})),~~~ r = 1,2,....,m,\\ \nonumber
&&_{0}I_{t}^{1-\eta}z(t)|_{t=0} = z_{0} \in Z,
\end{eqnarray}
where $_{0}D_{t}^{\eta}$ stands for the Riemann-Liouville fractional derivative of order $\eta$ with fixed lower limit as 0. $ A:D(A)\subseteq Z\rightarrow Z$ is densely defined and generates $C_0$-semigroup $T(t)(t>0)$. For each fixed $t$, $z(t)$ and $u(t)$ belong to Banach spaces $Z$ and $U$ respectively.  $B: L^q([0,a];U) \rightarrow L^q([0,a];Z)$ is a linear map. $h$ is a function from $[0,a] \times Z$ to $Z$. The points $p_r$ and $t_r$ satisfy the relation $0 = p_0<t_1 < p_1 < t_2 < ... < p_m < t_{m+1} = p_{m+1} = a$. The impulses start at points $t_r, ~r = 1,2,...,m$ and continue for the interval $(t_r,p_r]$. For $r=1,2,...m$, $ \psi_r$ are the impulsive functions to be discussed later. $z(t_r)=z(t_r^-) = lim_{\triangle \rightarrow 0^+}z(t_r - \triangle) ~\text{and} ~z(t_r^+) = lim_{\triangle \rightarrow 0^+}z(t_r + \triangle)$ denotes the left and right hand limit of $z(t)$ at $t_r, ~ r = 1,2,..,m$ respectively. \\
\indent This artefact 
is drafted as: Section \ref{2} gives the briefing for basic results and definitions. Section \ref{3} is dedicated to construction of mild solution. Results for the existence of solutions based on suitable assumptions are apparent in Section \ref{4}. Section \ref{5} accords with the sufficient assumptions and controllability conditions. Section \ref{6} presents examples in support of the theory presented. The article is winded up with concluding remarks and future scope in Section \ref{7}.

\section{Preliminary facts} \label{2}

This segment provides a quick referral to some fundamental concepts and definitions which are beneficial for the smooth study of the paper.\\
\indent Consider the Banach space 
\begin{eqnarray*}
 PC_{1-\eta}([0,a];Z) = &&\Big{ \{ }  z:z \in C \Big( \big( \cup_{r=0}^{m}(p_r,t_{r+1}) \big) \cup \big(\cup_{r=1}^{m}(t_r,p_r)\big); Z \Big),\\
&&~  z(p_r) = z(p_r^-) = lim_{\triangle \rightarrow 0^+} z(p_r - \triangle ) < \infty,~ r=1,2,...,m,\\
&&~ (t-p_r)^{1-\eta}\|z(t)\|<\infty,~ \text{for}~ t\in (p_r,p_{r+1}], ~r=0,1,...,m,\\
&&~ z(t_r) = z(t_r^-) = lim_{\triangle \rightarrow 0^+} z(t_r - \triangle ) < \infty, ~ r = 1,2,...,m\Big{ \} }.
\end{eqnarray*} 
Introduce the norm $\|z\|_{[0,a]}$ on $PC_{1-\eta}([0,a];Z)$ as 
$\|z\|_{[0,a]}=\text{max}_{r=0,1,..,m} \|z\|_r$
where
$\|z\|_r = \text{sup}_{t\in (p_r,p_{r+1}]} (t-p_r)^{1-\eta}\|z(t)\| ~\text{for}~ r=0,1,...,m $.

\begin{remark}
$PC_{1-\eta}([0,a];Z)$ is  a dense subset of $L^q([0,a];Z)$ if $q<\frac{1}{1-\eta}$.
\end{remark}
 Throughout this article, it is considered that $ M = \text{sup}_{t \in [0,a]} \|T(t)\| < \infty $ and $\tau = \underset{r= 0,1,...,m}{max} (t_{r+1} - p_r)$.\\
\indent Some definitions related to fractional integrals and derivatives are :
\begin{definition} 
The Riemann-Liouville $\eta^{th}$-order fractional integral is written in terms of the following integral
\begin{equation*}
_{t_0}I^{\eta}_{t}z(t)=\frac{1}{\Gamma \left( \eta \right)} \int_{t_0}^{t} (t-r)^{\eta - 1}z(r)dr, \hspace{6mm} \eta > 0,
\end{equation*}
where $\Gamma$ denotes the gamma function. 
\end{definition}

\begin{definition}The fractional $\eta^{th}$-order Riemann-Liouville derivative is defined by the following expression
\begin{equation*}
_{t_0}D^{\eta}_{t}z(t)=\frac{1}{\Gamma \left( n - \eta \right)} {\Bigg{(} \frac{d}{dt} \Bigg{)}}^n \int_{t_0}^{t} (t-r)^{n- \eta - 1}z(r)dr,
\end{equation*}
where $0 \leq n-1 < \eta < n$.
\end{definition}
\begin{definition}
A function of the complex variable $w$ defined by
\begin{eqnarray}
E_\eta (w) = \sum_{i=0}^{\infty} \dfrac{w^i}{\Gamma(\eta i + 1)} \nonumber
\end{eqnarray}
is known as the Mittag-Leffler function in one parameter.
\end{definition}

\begin{lemma}{\cite{kilbas:srivast}}
\begin{itemize}
\item[(i)] If $f \in C(t_0,a]$, then for any point $t\in (t_0,a]$ $$_{t_0}D_{t}^{\eta}(_{t_0}I_{t}^{\eta}f(t)) = f(t).$$
\item[(ii)] If $f \in C(t_0,a]$ and $_{t_0}I_{t}^{1-\eta}f(t) \in C(t_0,a]$, then for any point $t \in (t_0,a]$ $$_{t_0}I_{t}^{\eta}(_{t_0}D_{t}^{\eta}f(t)) = f(t) - \dfrac{_{t_0}I_{t}^{1-\eta}f(t)|_{t=t_0}}{\Gamma(\eta)}(t-t_0)^{\eta-1}.$$
\end{itemize}
\end{lemma}
\begin{proposition}{\cite{podlubny}}
The underneath holds true:
\begin{itemize}
\item[(i)]  For $ \alpha > 0$, $0<\eta<1$, 
\begin{eqnarray*}
_{t_0}I_{t}^{\eta}(t-t_0)^{\alpha-1} &=& \dfrac{\Gamma(\alpha)}{\Gamma(\alpha + \eta)}(t-t_0)^{\alpha + \eta - 1}, \\
_{t_0}D_{t}^{\eta}(t-t_0)^{\alpha -1} &=& \dfrac{\Gamma(\alpha)}{\Gamma(\alpha-\eta)}(t-t_0)^{\alpha - \eta - 1}.
\end{eqnarray*}
\item[(ii)] For $0<\eta <1$,
\begin{eqnarray*}
_{t_0}I_{t}^{\eta}(t-t_0)^{-\eta} &=& \Gamma(1-\eta),\\
_{t_0}D_{t}^{\eta}(t-t_0)^{\eta-1}&=&0.
\end{eqnarray*}
\end{itemize}
\end{proposition}

\begin{lemma}{\cite{zhou}}
The operator $T_\eta(t)$ possesses the underneath properties:
\begin{itemize}
\item[(i)] For every fixed $t\geq0$, operator $T_\eta(t)$ is linear and bounded, such that for any $z\in Z$,
\begin{eqnarray*}
\|T_\eta(t)z\| \leq \frac{M}{\Gamma(\eta)}{\|z\|}.
\end{eqnarray*}
where $M$ is a constant such that $\|T(t)\| \leq M$ for $t \in [0,\tau]$.
\item[(ii)] $T_\eta(t)(t\geq 0)$ is strongly continuous.
\end{itemize}
\end{lemma}

\section{Construction of mild solution} \label{3}
In order to construct the mild solution $z(t)$ for the system (\ref{system}), we proceed with the discussion below:\\
\textbf{Case 1:} For $t\in (0,t_1]$,
\begin{eqnarray*}
z(t) = t^{\eta-1}T_{\eta}(t)z_0 + \int_{0}^{t}(t-s)^{\eta-1}T_{\eta}(t-s)\big{[}Bu(s)+ h(s,z(s))\big{]}ds.
\end{eqnarray*}
\textbf{Case 2:} For $t \in (t_1,p_1]$,
\begin{eqnarray*}
z(t) &=& \psi_1(t,z(t_1^-)) \nonumber\\
&=& \psi_1\big{(}t, t_1^{\eta-1}T_{\eta}(t_1)z_0 + \int_{0}^{t_1}(t_1-s)^{\eta-1}T_{\eta}(t_1-s)\big{[}Bu(s)+ h(s,z(s))\big{]}ds\big{)}
\end{eqnarray*}
\textbf{Case 3:} For $t \in (p_1,t_2]$,
\begin{align*}
_{0}D_{t}^{\eta}z(t) ={}& \dfrac{1}{\Gamma(1-\eta)} \dfrac{d}{dt} \int_{0}^{t}(t-s)^{1-\eta - 1} z(s) ds\\
={}& \dfrac{1}{\Gamma(1-\eta)} \dfrac{d}{dt} \int_{0}^{t_1}(t-s)^{-\eta} z(s) ds + \dfrac{1}{\Gamma(1-\eta)} \dfrac{d}{dt} \int_{t_1}^{p_1}(t-s)^{-\eta} \psi_1(s,z(t_1^-)) ds\\
&\quad+ \dfrac{1}{\Gamma(1-\eta)} \dfrac{d}{dt} \int_{p_1}^{t}(t-s)^{-\eta} z (s) ds\\
={}& \dfrac{-\eta}{\Gamma(1-\eta)}\int_{0}^{t_1}\dfrac{z(s)}{(t-s)^{1+\eta}}ds + \dfrac{-\eta}{\Gamma(1-\eta)}\int_{t_1}^{p_1}\dfrac{\psi_1(s,z(t_1^-))}{(t-s)^{1+\eta}}ds +~ _{p_1}D_{t}^{\eta}z(t)\\
={}& - \phi_1(t,z(t)) +~ _{p_1}D_{t}^{\eta}z(t)
\end{align*}
Thus,
\begin{eqnarray*}
&&_{p_1}D_{t}^{\eta}z(t) =  Az(t) + Bu(t) + h(t,z(t)) +  \phi_1(t,z(t)) \\
&&_{p_1}I_{t}^{1-\eta}z(t)|_{t=p_1} = \psi_1(p_1,z(t_{1}^{-}))
\end{eqnarray*}
Hence, for $t \in (p_1,t_2]$
\begin{eqnarray*}
z(t) &=& (t-p_1)^{\eta-1}T_{\eta}(t-p_1)\psi_1(p_1,z(t_{1}^{-})) \nonumber\\
&&~~+ \int_{p_1}^{t}(t-s)^{\eta-1}T_{\eta}(t-s)\big{[}Bu(s)+h(s,z(s))+ \phi_1(s,z(s))\big{]}ds
\end{eqnarray*}
Continuing this process for each $r=2,3,\ldots,m$ and taking lower limit of Riemann-Liouville derivative as $p_r$ for each interval $(p_r,t_{r+1}]$, we define the mild solution as follows:

\begin{definition}
A function $z \in PC_{1-\eta}([0,a];Z)$ is called a mild solution of the system (\ref{system}) if it satisfies the following integral equation:
\begin{eqnarray}
z(t) = 
\begin{cases}
     (t-p_r)^{\eta-1}T_{\eta}(t-p_r)\psi_r(p_r,z(t_{r}^{-}))+ \displaystyle{\int_{p_r}^{t}}(t-s)^{\eta-1}T_{\eta}(t-s)\big[ Bu(s)\\
 +h(s,z(s))+ \phi_r(s,z(s))\big]ds,  \quad \text{for}\ t\in (p_r,t_{r+1}], ~~~r = 0,1,...,m \\

      \psi_r(t,z(t_{r}^{-})),  ~~~~~~~~~~~~~~~~\text{for} \ t \in (t_r,p_r], ~~~ r = 1,2,...,m
    \end{cases}
 \label{mildsol}
\end{eqnarray}
where

\begin{eqnarray}
&\psi_0 = z_0 , \nonumber \\
&\phi_0 = 0,   \label{phidef}\\
&\phi_r(t,z(t)) = \dfrac{\eta}{\Gamma(1-\eta)}\displaystyle{\sum_{k=0}^{r-1}}\bigg{(} \displaystyle{\int_{p_k}^{t_{k+1}}}\frac{z(s)}{(t-s)^{1+\eta}} ds + \displaystyle{\int_{t_{k+1}}^{p_{k+1}}} \dfrac{\psi_{k+1}(s,z(t_{k+1}^-))}{(t-s)^{1+ \eta}} ds \bigg{)}, \nonumber
\end{eqnarray}
for $ r=1,2,...,m$; and
\begin{eqnarray}
T_{\eta}(t)&=&\eta \int_{0}^{\infty} \theta \xi_{\eta}(\theta)T(t^{\eta} \theta)d\theta ,\\
\xi_{\eta}(\theta)&=&\dfrac{1}{\eta} \theta^{-1-\frac{1}{\eta}}\varpi_\eta{(\theta^{-\frac{1}{\eta}})}, \nonumber\\
\varpi_\eta{(\theta)}&=& \dfrac{1}{\pi}{\sum_{n=1}^{\infty}}(-1)^{n-1}{\theta}^{-n\eta-1}\frac{\Gamma(n\eta + 1)}{n!}sin(n\pi \eta), \hspace{4mm} \theta \in (0,\infty), \nonumber
\end{eqnarray}
and $\xi_{\eta}$ is a probability density function defined on $(0,\infty)$, that is,\\

$\xi_{\eta}(\theta) \geq 0$\; and\; $\int_{0}^{\infty} \xi_{\eta}(\theta)d\theta = 1$.
\end{definition}

\begin{definition}
Let $z(t,u)$ be a mild solution of the system (\ref{system}) at time t corresponding to a control $u(\cdot)\in L^q([0,a];U)$. The set $K_a(h)=\lbrace z(a,u) \in Z; u(\cdot)\in L^q([0,a];U \rbrace$ is called the reachable set for final time a. If $K_a(h)$ becomes dense in Z, the system (\ref{system}) is approximately controllable on [0,a].

\end{definition}

\section{Existence of mild solution} \label{4}
This segment establishes the existence and uniqueness of mild solution for the system (\ref{system}) utilizing the Banach fixed point theorem. The results are  based on the below mentioned assumptions:
\begin{itemize}
\item[(H0)] $q \in (\frac{1}{\eta}, \frac{1}{1-\eta})$.
\item[(H1)] A constant $\kappa>0$ exists in a way satisfying \\
$$ \|h(t,z)-h(t,y)\|_Z \leq \kappa \|z-y\|_Z~~\forall ~z,y \in Z. $$
\item[(H2)] A function $\varsigma(.) ~\mbox{in}~ L^q([0,a];\mathbb{R^+}), ~q>\frac{1}{\eta}$, and a constant $d>0$ exists such that\\
$$\|h(t,z)\|_Z \leq \varsigma(t) + d (t-p_r)^{1-\eta}\|z\|_{Z}$$ for a.e. $t\in (p_r,p_{r+1}]; r=0,1,...,m ~ \mbox{and}~ \forall ~ z \in Z$. 
\item[(H3)] For $r=1,2,...,m$, the impulsive functions $\psi_{r}$ defined from $[t_r,p_r] \times Z$ to $Z$ are continuous and there exist constants $b_r \in (0,1]$, such that \\
$$\| \psi_r(t,z) - \psi_r(t,y) \|_Z \leq b_r \|z-y\|_Z.$$
\item[(H4)] The constant $\nu < 1$, ~
where
\begin{eqnarray*}
\nu &=& \underset{r=1,2,...,m}{max}~ \dfrac{Mb_r}{\Gamma(\eta) (t_r - p_{r-1})^{1-\eta}} + \dfrac{M\kappa\Gamma(\eta)}{\Gamma(2\eta)}(t_{r+1} - p_r)^\eta \\
&& +\dfrac{M}{\Gamma(1+\eta) \Gamma(1-\eta)} \left( \displaystyle{\sum_{k=0}^{r-1}}\Big{(}\frac{t_{k+1}-{p_k}}{p_r - t_{k+1}}\Big{)}^\eta 
+ \sum_{k=0}^{r-2} \dfrac{b_{k+1}(t_{r+1}-p_r)}{(t_{k+1}-p_k)^{1-\eta}(p_r - p_{k+1})^\eta} \right)\\
&& + \dfrac{ Mb_r (t_{r+1} - p_r)^{1-\eta}}{(t_{r} - p_{r-1})^{1-\eta}}.
\end{eqnarray*}
\end{itemize}
\begin{theorem}
The nonlinear system (\ref{system}) admits a unique mild solution in $PC_{1-\eta}([0,a];Z)$ for each control $u(\cdot) \in L^q([0,a];U)$ under the assumptions $(H0)-(H4)$.
\end{theorem}

%
\begin{proof}
Consider the operator $G$ as 
\begin{eqnarray}
(Gz)(t) = 
 \begin{cases}
     (t-p_r)^{\eta-1}T_{\eta}(t-p_r)\psi_r(p_r,z(t_{r}^{-})) + \displaystyle{\int_{p_r}^{t}}(t-s)^{\eta-1}T_{\eta}(t-s)\big{[}Bu(s)\\
+h(s,z(s))+ \phi_r(s,z(s))\big{]}ds, ~~\text{for}\ t\in (p_r,t_{r+1}], ~~~r = 0,1,\ldots,m \\
    
      \psi_r(t,G(z(t_{r}^{-}))),  ~~~~~~~~~~~~~~~~\text{for} \ t \in (t_r,p_r], ~~~r=1,\ldots,m.
 \end{cases}
\end{eqnarray}
\indent First, it is evident that $G$ maps $PC_{1-\eta}([0,a];Z)$ into itself.
It is now required to prove $G$ as a contraction operator on $PC_{1-\eta}([0,a];Z)$.\\
Let $z,y \in PC_{1-\eta}([0,a];Z)$, then\\
\\
\textbf{Case 1:}~~ For $t \in (p_r,t_{r+1}]$, ~~$r= 0,1,2,\ldots,m$.
\begin{eqnarray}
&&(t-p_r)^{1-\eta} \big{\|} (Gz)(t) - (Gy)(t) \big{\|} \nonumber \\
&& \quad \leq \big{\|} T_\eta (t-p_r)[\psi_r(p_r,z(t_{r}^{-})) - \psi_r(p_r,y(t_{r}^{-})] \big{\|}  \nonumber \\
&&\quad \quad + (t-p_r)^{1-\eta} \int_{p_r}^{t} (t-s)^{\eta-1} \big{\|} T_\eta(t-s)\big{[} \big{(} h(s,z(s)) - h(s,y(s)) \big{)} \big{]} \big{\|} ds \nonumber\\
&&\quad \quad + (t-p_r)^{1-\eta} \int_{p_r}^{t} (t-s)^{\eta-1} \big{\|} T_\eta(t-s)\big{[} \big{(} \phi_r(s,z(s)) - \phi_r(s,y(s)) \big{)} \big{]} \big{\|} ds \nonumber  \\
&&\quad \leq  \frac{M}{\Gamma(\eta)} \Big{\|} [\psi_r(p_r,z(t_{r}^{-}) - \psi_r(p_r,y(t_{r}^{-})] \Big{\|}_Z + \frac{M\kappa (t-p_r)^{1-\eta}}{\Gamma(\eta)} \int_{p_r}^{t} (t-s)^{\eta - 1}  \|z(s) - y(s)\|_Z ds \nonumber \\
&&\quad \quad +\frac{M (t-p_r)^{1-\eta}}{\Gamma(\eta)} \int_{p_r}^{t} (t-s)^{\eta - 1} \big{\|} \phi_r(s,z(s)) - \phi_r(s,y(s)) \big{\|}_Z ds.  \label{halfeq}
\end{eqnarray}
Solving $ (t-p_r)^{1-\eta} \int_{p_r}^{t} (t-s)^{\eta - 1} \big{\|} \phi_r(s,z(s)) - \phi_r(s,y(s)) \big{\|} ds$ seperately for $r=1,2,\ldots,m$ as for  we proceed with
\begin{eqnarray*}
\phi_r(s,z(s)) = \dfrac{\eta}{\Gamma(1-\eta)}\displaystyle{\sum_{k=0}^{r-1}}\bigg{(} \displaystyle{\int_{p_k}^{t_{k+1}}}\frac{z(s_1)}{(s-s_1)^{1+\eta}}ds_1 + \displaystyle{\int_{t_{k+1}}^{p_{k+1}}} \frac{\psi_{k+1}(s_1,z(t_{k+1}^-))}{(s-s_1)^{1+ \eta}}ds_1 \bigg{)},~ r=1,2,\ldots,m.
\end{eqnarray*}
and for $r=0$, $\phi_0 = 0$.\\
So, basing our following evaluations for $t\in (p_r,t_{r+1}], ~ r=1,\ldots,m$,
\begin{eqnarray}
\int_{p_k}^{t_{k+1}} \dfrac{\|z(s_1)-y(s_1)\|}{{(s-s_1)}^{\eta+1}} ds_1 &=& \int_{p_k}^{t_{k+1}} \dfrac{{(s_1 - p_k)}^{1-\eta} \|z(s_1)-y(s_1)\|}{{(s_1 - p_k)}^{1-\eta}{(s-s_1)}^{\eta+1}} ds_1 \nonumber \\
& \leq & \dfrac{{(t_{k+1} - p_k)}^\eta}{\eta {(s-t_{k+1})}^{\eta + 1} }\| z-y \|_k  \label{singint1}
\end{eqnarray}
Using (\ref{singint1}) and $ t-t_{k+1} > t - p_r, ~ k=0,1,\ldots,r-1$, we obtain
\begin{eqnarray}
&&{(t-p_r)}^{1-\eta} \int_{p_r}^{t} \int_{p_k}^{t_{k+1}} {(t-s)}^{\eta-1} \dfrac{\|z(s_1)-y(s_1)\|}{{(s-s_1)}^{\eta+1}} ds_1 ds \nonumber \\
&&\quad \leq {(t-p_r)}^{1-\eta} \dfrac{{(t_{k+1} - p_k)}^\eta \| z-y \|_k}{\eta }  \int_{p_r}^{t}  \dfrac{1}{{(t-s)}^{1-\eta}{(s-t_{k+1})}^{\eta + 1}} ds \nonumber \\
&&\quad = {(t-p_r)}^{1-\eta}\dfrac{{(t_{k+1} - p_k)}^\eta \| z-y \|_k}{\eta} \dfrac{{(t-p_r)}^\eta}{\eta {(p_r - t_{k+1})}^\eta (t - t_{k+1})} \nonumber \\
&&\quad \leq \dfrac{\|z-y\|_k}{\eta^2} {\left( \dfrac{t_{k+1} - p_k}{p_r - t_{k+1}} \right)}^\eta . \label{doubint1}
\end{eqnarray}
Next, for $k=0,1,\ldots,r-1$ and $r=1,2,\ldots,m$, we have
\begin{eqnarray}
&&\int_{t_{k+1}}^{p_{k+1}} \dfrac{\| \psi_{k+1}(s_1,z(t_{k+1})) - \psi_{k+1}(s_1,y(t_{k+1}))\|}{{(s-s_1)}^{\eta + 1}} ds_1 \nonumber \\
&&\quad \leq b_{k+1} \int_{t_{k+1}}^{p_{k+1}} \dfrac{{(t_{k+1}-p_k)}^{1-\eta} \| z(t_{k+1})-y(t_{k+1})\|}{{(t_{k+1}-p_k)}^{1-\eta}{(s-s_1)}^{\eta + 1}} ds_1 \nonumber \\
&&\quad \leq  \dfrac{b_{k+1} \| z-y\|_{k} }{\eta{(t_{k+1}-p_k)}^{1-\eta}} \Bigg{[} \dfrac{1}{{(s-p_{k+1})}^\eta} - \dfrac{1}{{(s-t_{k+1})}^{\eta}} {\Bigg{]}} \nonumber \\
&&\quad \leq \dfrac{b_{k+1} \| z-y\|_{k} }{\eta{(t_{k+1}-p_k)}^{1-\eta}} . \dfrac{1}{{(s-p_{k+1})}^\eta}. \label{singint2}
\end{eqnarray}
Using (\ref{singint2}) for $k=0,1,2,\ldots,r-2$ and $r=2,3,\ldots,m$, we obtain
\begin{eqnarray}
&&{(t-p_r)}^{1-\eta} \int_{p_r}^{t} \int_{t_{k+1}}^{p_{k+1}} {(t-s)}^{\eta-1} \dfrac{\| \psi_{k+1}(s_1,z(t_{k+1})) - \psi_{k+1}(s_1,y(t_{k+1}))\|}{{(s-s_1)}^{\eta + 1}} ds_1 ds \nonumber \\
&&\quad \leq  \dfrac{b_{k+1} \| z-y\|_{k} }{\eta {(t_{k+1} - p_k)}^{1-\eta}} {(t-p_r)}^{1-\eta} \int_{p_r}^{t} \dfrac{1}{{{(s-p_{k+1})}^{\eta}}{(t-s)}^{1-\eta}} ds \nonumber \\
&&\quad \leq \dfrac{b_{k+1} \| z-y\|_{k} }{\eta {(t_{k+1} - p_k)}^{1-\eta}} {(t-p_r)}^{1-\eta}. \dfrac{{(t-p_r)}^{\eta}}{\eta {(p_r - p_{k+1})}^\eta} \nonumber \\
&&\quad \leq \dfrac{b_{k+1}{(t_{r+1} - p_r)}}{\eta^2{(t_{k+1} - p_k)}^{1-\eta}{(p_r - p_{k+1})}^\eta}.\| z-y\|_{k} \label{doubint2}
\end{eqnarray}
and for $k=r-1$ and $r=1,2,\ldots,m$, it is 
\begin{eqnarray}
&&{(t-p_r)}^{1-\eta} \int_{p_r}^{t} \int_{t_{r}}^{p_{r}} {(t-s)}^{\eta-1} \dfrac{\| \psi_{r}(s_1,z(t_{r})) - \psi_{r}(s_1,y(t_{r}))\|}{{(s-s_1)}^{\eta + 1}} ds_1 ds \nonumber\\
&&\quad \leq {(t-p_r)}^{1-\eta} \dfrac{b_{r} \| z-y\|_{r-1} }{\eta{(t_{r}-p_{r-1})}^{1-\eta}} \int_{p_r}^{t} \dfrac{1}{{(t-s)}^{1-\eta}{{(s-p_{r})}^\eta}} ds \nonumber\\
&&\quad = {(t-p_r)}^{1-\eta} \dfrac{b_{r} \| z-y\|_{r-1} }{\eta{(t_{r}-p_{r-1})}^{1-\eta}} B(\eta,1-\eta) \nonumber \\
&&\quad \leq {(t_{r+1}-p_r)}^{1-\eta} \dfrac{b_{r} \| z-y\|_{r-1} }{\eta{(t_{r}-p_{r-1})}^{1-\eta}} \Gamma(\eta)\Gamma(1-\eta). \label{doubint3}
\end{eqnarray}
Combining equations (\ref{phidef}), (\ref{doubint1}), (\ref{doubint2}), and (\ref{doubint3}), we obtain
\begin{eqnarray}
&&(t-p_r)^{1-\eta} \int_{p_r}^{t} (t-s)^{\eta - 1} \big{\|} \phi_r(s,z(s)) - \phi_r(s,y(s)) \big{\|} ds \nonumber \\
&&\quad \leq \dfrac{\|z-y\|_{[0,a]}}{\eta \Gamma(1-\eta)} \Bigg{(} \sum_{k=0}^{r-1} {\left( \dfrac{t_{k+1} - p_k}{p_r - t_{k+1}} \right)}^\eta + \sum_{k=0}^{r-2} \dfrac{b_{k+1}{(t_{r+1} - p_r)}}{{(t_{k+1} - p_k)}^{1-\eta}{(p_r - p_{k+1})}^\eta} \nonumber \\
&&\quad \quad + \dfrac{b_r (t_{r+1}-p_r)^{1-\eta}}{(t_{r}-p_{r-1})^{1-\eta}} \Gamma(1+\eta) \Gamma(1-\eta) \Bigg{)}. \label{normphi}
\end{eqnarray}

Applying equation (\ref{normphi}) in equation (\ref{halfeq}), we proceed as 
\begin{eqnarray}
&&(t-p_r)^{1-\eta} \big{\|} (Gz)(t) - (Gy)(t) \big{\|} \nonumber \\
&&\quad \leq \dfrac{Mb_r}{\Gamma(\eta)} \dfrac{\| z - y \|_{r-1}}{ (t_r - p_{r-1})^{1-\eta}} + \dfrac{M\kappa(t-p_r)^{1-\eta}}{\Gamma(\eta)}\|z-y\|_r \int_{p_r}^{t} (t-s)^{\eta - 1} (s-p_r)^{\eta - 1} ds\nonumber \\
&&\quad \quad + \dfrac{M \| z - y \|_{[0,a]}}{\Gamma(1 + \eta) \Gamma(1 - \eta)} \Bigg{(} \sum_{k=0}^{r-1} {\bigg{(} \dfrac{t_{k+1}-p_k}{p_r - t_{k+1}} \bigg{)}}^\eta + \sum_{k=0}^{r-2} \dfrac{b_{k+1}(t_{r+1}-p_r)}{(t_{k+1}-p_k)^{1-\eta}(p_r - p_{k+1})^\eta} \Bigg{)} \nonumber \\
&&\quad \quad + \dfrac{Mb_r (t_{r+1}-p_r)^{1-\eta} }{(t_{r} - p_{r-1})^{1-\eta}} \| z - y \|_{[0,a]} \nonumber \\
&&\quad \leq \Bigg{[}\dfrac{Mb_r}{\Gamma(\eta)(t_r - p_{r-1})^{1-\eta}} + \dfrac{M\kappa \Gamma(\eta)}{\Gamma(2\eta)} (t_{r+1} - p_r)^\eta  \nonumber \\
&&\quad \quad + \dfrac{M}{\Gamma(1 + \eta) \Gamma(1 - \eta)} \Bigg{(} \sum_{k=0}^{r-1} {\bigg{(} \dfrac{t_{k+1}-p_k}{p_r - t_{k+1}} \bigg{)}}^\eta + \sum_{k=0}^{r-2} \dfrac{b_{k+1}(t_{r+1}-p_r)}{(t_{k+1}-p_k)^{1-\eta}(p_r - p_{k+1})^\eta} \Bigg{)} \nonumber \\
&& \quad \quad + \dfrac{ Mb_r (t_{r+1} - p_r)^{1-\eta}}{(t_{r} - p_{r-1})^{1-\eta}} \Bigg{]} \| z - y \|_{[0,a]} \nonumber \\
&& \quad = \nu \| z - y \|_{[0,a]}. \label{c1}
\end{eqnarray}
\textbf{Case 2:} ~~For $t \in (t_r,p_r]$,~~$r= 1,2,...,m$
\begin{eqnarray}
&&(t-p_{r-1})^{1-\eta} \big{\|} (Gz)(t) - (Gy)(t) \big{\|} \nonumber \\
&&\quad = (t-p_{r-1})^{1-\eta} \| \psi_r(t,G(z(t_{r}^{-})) - \psi_r(t,G(y(t_{r}^{-}))\| \nonumber \\
&& \quad\leq b_r (t-p_{r-1})^{1-\eta} \| G(z(t_r)) - G(y(t_r)) \|_Z \nonumber \\
&& \quad \leq b_r \Bigg{[}\dfrac{Mb_r}{\Gamma(\eta)(t_r - p_r)^{1-\eta}} + \dfrac{M\kappa \Gamma(\eta)}{\Gamma(2\eta)} (t_{r+1} - p_r)^\eta  \nonumber \\
&&\quad \quad + \dfrac{M}{\Gamma(1 + \eta) \Gamma(1 - \eta)} \Bigg{(} \sum_{k=0}^{r-1} {\bigg{(} \dfrac{t_{k+1}-p_k}{p_r - t_{k+1}} \bigg{)}}^\eta  + \sum_{k=0}^{r-2} \dfrac{b_{k+1}(t_{r+1}-p_r)}{(t_{k+1}-p_k)^{1-\eta}(p_r - p_{k+1})^\eta} \Bigg{)}  \nonumber \\
&& \quad \quad + \dfrac{ Mb_r (t_{r+1} - p_r)^{1-\eta}}{(t_{r} - p_{r-1})^{1-\eta}} \Bigg{]} \| z - y \|_{[0,a]} \nonumber \\
&& \quad \leq \nu \|z-y\|_{[0,a]}. \label{c2}
\end{eqnarray}
Therefore, on combining equations (\ref{c1}) and (\ref{c2}), we obtain 
\begin{eqnarray*}
\|Gz-Gy\|_{[0,a]} \leq \nu \| z - y \|_{[0,a]} 
\end{eqnarray*}
Thus, $G$ is   contraction and it is evident through Banach fixed point theorem that $G$ possess a unique fixed point $z(\cdot)$ in $PC_{1-\eta}([0,a];Z)$ which serves as the requisite solution of system (\ref{system}).
\end{proof}

\section{Controllability Results} \label{5}
We define the following operators:\\
The Nemytskii operator 
corresponding to the nonlinear function $h$ for $t \in (p_r,t_{r+1}]$, $r=0,1,...,m$, be denoted by
$\Omega_h : C_{1-\eta}([p_r,t_{r+1}];Z)\rightarrow L^q([p_r,t_{r+1}];Z)$ is defined by 
\begin{eqnarray*}
\Omega_h(z)(t)=h(t,z(t)),
\end{eqnarray*}
where Banach space $C_{1-\eta}([p_r,t_{r+1}];Z)= \{z: (t-p_r)^{1-\eta}z(t) \in C([p_r,t_{r+1}];Z)\}$.
The Nemytskil operator 
corresponding to the function $\phi_r$  be denoted by
$\Omega_{\phi_r} : C_{1-\eta}([p_r,t_{r+1}];Z)\rightarrow L^q([p_r,t_{r+1}];Z)$ is defined by 
\begin{eqnarray*}
\Omega_{\phi_r}(z)(t)=\phi_r(t,z(t)).
\end{eqnarray*}
The bounded linear operator $\mathbb{F}: L^q([p_r,t_{r+1}];Z)\rightarrow Z$ as 
\begin{eqnarray*}
\mathbb{F}z=\int_{p_r}^{t_{r+1}}(t_{r+1}-s)^{\eta - 1}T_\eta(t_{r+1}-s)z(s)ds.
\end{eqnarray*}
The following assumptions are made to prove approximate controllability of the system(\ref{system})
\begin{itemize}
\item[(H5)] There is a constant $\tilde{\kappa} > 0$ such that \\
$$ \|h(t,z) - h(t,y)\|_Z \leq \tilde{\kappa}  {(t-p_r)^{1-\eta}} \|z-y\|_Z$$ 
$~~ \forall~~ z,y \in Z$ and $t\in (p_r,p_{r+1}]$, for $r=0,1,\ldots,m$.

\item[(H6)] For $r=1,2,\ldots,m$, the impulsive functions $\psi_{r}$ defined from $[t_r,p_r] \times Z$ to $Z$ are continuous and there exist constants $c_r \in (0,1]$, such that 
$$\| \psi_r(t,z(t)) \| \leq c_r \|z(t)\| ~\mbox{for each}~ t\in [t_r,p_r] ~\mbox{and}~ z\in Z.$$
\item[(H7)] $\mu E_\eta (M \tilde{\kappa}\tau) < 1$.
\item[(H8)] For any $\epsilon > 0$ and $\vartheta^r(\cdot) \in L^q([p_r,t_{r+1}];Z)$, there exists a $u^r(\cdot) \in L^q([0,a];Z)$ and hence $u^r(\cdot) \in L^q([p_r,t_{r+1}];U)$ satisfying \\
$$\| \mathbb{F}\vartheta^r - \mathbb{F}Bu^r\|_Z < \epsilon $$
$$\|Bu^r(\cdot)\|_{L^q([p_r,t_{r+1}];Z)} < \aleph \|\vartheta^r(\cdot)\|_{L^q} $$ where $\aleph$ is a constant independent of $\vartheta^r(\cdot) \in L^q([p_r,t_{r+1}];Z)$ and 
\begin{eqnarray}
\aleph \Bigg[\tilde{\kappa}(t_{r+1}-p_r)^{\frac{1}{q}} + \dfrac{1}{\Gamma(1-\eta)}\sum_{k=0}^{r-1}\bigg(\dfrac{(t_{k+1}-p_k)^\eta (t_{r+1}-p_r)^{\frac{1}{q}}}{(p_r - t_{k+1})^{1+\eta}} \nonumber\\
 + \dfrac{b_{k+1}(t_{r+1}-p_{k+1})^{\frac{1}{q}-\eta}}{(t_{k+1}-p_k)^{1-\eta}(1-\eta q)^{\frac{1}{q}}} \bigg) \Bigg]\dfrac{\varrho E_\eta (M\tilde{\kappa}\tau)}{1-\mu E_\eta (M\tilde{\kappa}\tau)} < 1 \label{mainass}
\end{eqnarray}
for each $r = 1,2,...,m$.
\end{itemize}
%
%
%
%
%
\begin{lemma}
Under the assumptions ${(H0),(H2),(H3)}$ and ${(H5)}-{(H7)}$, any mild solution of the system (\ref{system}) satisfies the following:
\begin{itemize}
\item[(i)] $\|z(\cdot,u)\|_{[0,a]} \leq  \Lambda E_\eta (Md\tau) ~~~~ for~ any~ u(\cdot) \in L^q([0,a];U),$\\
\item[(ii)] $\| z (\cdot,u) - y(\cdot,v)\|_{[0,a]} \leq \dfrac{\varrho E_\eta (M\tilde{\kappa} \tau)}{1-\mu E_\eta (M\tilde{\kappa} \tau)} \|Bu-Bv\|_{L^q}$,
\end{itemize}
where
\begin{eqnarray*} 
&&\Lambda = \underset{r= 0,1,...,m}{max} \frac{M}{\Gamma(\eta)} \Bigg{[} \dfrac{c_r}{{(t_r - p_{r-1})^{1-\eta}}} \|z\|_{r-1} + {\Big{(} \frac{q-1}{q\eta-1}\Big{)}}^{\frac{q-1}{q}} (t_{r+1}-p_r)^{1-\frac{1}{p}} \Big{(}\|Bu\|_{L^q} + \| {\varsigma} \|_{L^q} \Big{)} \\
&& \quad \quad + \dfrac{1}{\eta \Gamma(1-\eta)} \Bigg{(} \sum_{k=0}^{r-1} {\bigg{(} \dfrac{t_{k+1}-p_k}{p_r - t_{k+1}} \bigg{)}}^\eta \|z\|_k + (t_{r+1}-p_r)\sum_{k=0}^{r-2} \dfrac{c_{k+1}}{(p_r - p_{k+1})^\eta} \Bigg{)}\\
&& \quad \quad + c_r \Gamma(\eta) (t_{r+1}-p_r)^{1-\eta} \Bigg{]}.\\
&&\varrho = \dfrac{M}{\Gamma(\eta)}{\bigg{(}\dfrac{q-1}{q \eta - 1 } \bigg{)}} ^ {\frac{q-1}{q}}\tau^{1 - \frac{1}{q}}\\
%
&& \mu = \underset{r=1,2,...,m}{max} \dfrac{Mb_r}{\Gamma(\eta) (t_r - p_{r-1})^{1-\eta}} 
 +\dfrac{M}{\Gamma(1+\eta) \Gamma(1-\eta)} \Bigg{(} \displaystyle{\sum_{k=0}^{r-1}}\Big{(}\frac{t_{k+1}-{p_k}}{p_r - t_{k+1}}\Big{)}^\eta \\
&& \quad \quad + \sum_{k=0}^{r-2} \dfrac{b_{k+1}(t_{r+1}-p_r)}{(t_{k+1}-p_k)^{1-\eta}(p_r - p_{k+1})^\eta} \Bigg{)} + \dfrac{ Mb_r (t_{r+1} - p_r)^{1-\eta}}{(t_{r} - p_{r-1})^{1-\eta}}.
%
\end{eqnarray*}
\end{lemma}

\begin{proof}(i) Let $z \in PC_{1-\eta}{([0,a];Z)}$ be a mild solution of the system $(1)$ corresponding to the control $u(\cdot) \in L^q([0,a];U)$, then \\
\begin{eqnarray*}
z(t) = 
\begin{cases}
     (t-p_r)^{\eta-1}T_{\eta}(t-p_i)\psi_i(p_r,z(t_{r}^{-}))+ \displaystyle{\int_{p_r}^{t}}(t-s)^{\eta-1}T_{\eta}(t-s)\big[ Bu(s)\\~~\\
  +h(s,z(s))+ \phi_r(s,z(s))\big]ds,  \quad \text{for}\ t\in (p_r,t_{r+1}], ~~~r = 0,1,...,m \\

     \\
      \psi_r(t,z(t_{r}^{-})),  ~~~~~~~~~~~~~~~~\text{for} \ t \in (t_r,p_r], ~~~ r = 1,2,...,m
    \end{cases}
\end{eqnarray*}
\textbf{Case 1:} For $t \in (p_r,t_{r+1}], r=0,1,2,\ldots,m$
\begin{eqnarray}
&&(t-p_r)^{1-\eta} \|z(t)\| \nonumber \\
&&~~ \leq \|T_\eta (t-p_r)\psi_r(p_r,z(t_r^-))\| + (t-p_r)^{1-\eta}\int_{p_r}^{t} (t-s)^{\eta - 1} \big{\|} T_\eta (t-s) \big{[}Bu(s)\nonumber \\
&&~~~~~~+h(s,z(s))+\phi_r(s,z(s))\big{]} \big{\|} ds \nonumber \\
&&~~\leq \dfrac{M}{\Gamma(\eta)} \Bigg{[} c_r \|z(t_r)\| + (t-p_r)^{1-\eta} \int_{p_r}^{t} (t-s)^{\eta - 1} \big{[} \| Bu(s)\| + \|h(s,z(s))\| +\| \phi_r(s,z(s))\| \big{]}ds \Bigg{]} \nonumber \\
&&~~\leq \dfrac{M}{\Gamma(\eta)} \Bigg{[}c_r \dfrac{(t_r - p_{r-1})^{1-\eta} \|z(t_r)\|}{(t_r - p_{r-1})^{1-\eta}} + (t-p_r)^{1-\eta} \int_{p_r}^{t} (t-s)^{\eta - 1} \big{[} \|Bu(s)\| + \| \varsigma(s)\| + d (s-p_r)^{1-\eta}\|z(s)\| \big{]} ds \nonumber \\
&&~~~~~~+  (t-p_r)^{1-\eta}\int_{p_r}^{t} (t-s)^{\eta - 1} \| \phi_r(s,z(s)) \| ds \Bigg{]}\nonumber \\
&&~~ \leq \dfrac{M}{\Gamma(\eta)} \Bigg{[} \dfrac{c_r}{{(t_r - p_{r-1})^{1-\eta}}} \|z\|_{r-1} + {\bigg{(}\dfrac{q-1}{q \eta - 1 } \bigg{)}} ^ {\frac{q-1}{q}} (t_{r+1} - p_r)^ {1 - \frac{1}{q}}(\|Bu\|_{L^q} + \| \varsigma\|_{L^q})\nonumber \\
&&~~~~~~+(t-p_r)^{1-\eta}\int_{p_r}^{t} (t-s)^{\eta - 1} d (s-p_r)^{1-\eta}\|z(s)\|ds \nonumber \\
&&~~~~~~ + \dfrac{1}{\eta \Gamma(1-\eta)} \Bigg{(} \sum_{k=0}^{r-1} {\bigg{(} \dfrac{t_{k+1}-p_k}{p_r - t_{k+1}} \bigg{)}}^\eta \|z\|_k + \sum_{k=0}^{r-2} \dfrac{c_{k+1}(t_{r+1}-p_r)}{(p_r - p_{k+1})^\eta} \Bigg{)} \nonumber \\
&&~~~~~~ + c_r \Gamma(\eta) (t_{r+1}-p_r)^{1-\eta} \Bigg{]}. \nonumber 
\end{eqnarray}
Thus, 
\begin{eqnarray*}
(t-p_r)^{1-\eta} \|z(t)\| \leq \Lambda + \dfrac{Md(t_{r+1} - p_r)^{1-\eta}}{\Gamma(\eta)}\int_{p_r}^{t} (t-s)^{\eta - 1} (s-p_r)^{1-\eta}\|z(s)\|ds 
\end{eqnarray*}
Using generalised Gronwall's inequality (\cite{zhou}), it follows that 
\begin{eqnarray}
(t-p_r)^{1-\eta} \|z(t)\| \leq \Lambda E_{\eta} (Md\tau) \label{2c1}
\end{eqnarray}
\textbf{Case 2:} For $t \in (t_r,p_r], r=1,2,...,m$
\begin{eqnarray}
(t-p_{r-1})^{1-\eta} \|z(t)\| &=& (t-p_{r-1})^{1-\eta} \| \psi_r(t,z(t_r^-))\| \nonumber \\
&\leq & c_r (t-p_{r-1})^{1-\eta} \|z(t_r)\| \nonumber \\
&\leq & \Lambda E_{\eta} (Md\tau)  \label{2c2}
\end{eqnarray}
Clearly, on combining the results for both the cases discussed above, i.e., from eqns. (\ref{2c1}) and (\ref{2c2}), it is proved that 
\begin{eqnarray}
\|z\|_{[0,a]} \leq \Lambda E_{\eta} (Md\tau)
\end{eqnarray}
(ii) Similarly, proceeding to prove the next inequality, we have \\
\textbf{Case 1:} For $t \in (p_r,t_{r+1}], r=0,1,2,...,m$
\begin{eqnarray*}
&&(t-p_r)^{1-\eta} \|z(t) - y(t)\| \\
&&~~\leq \|T_\eta (t-p_r)\psi_r(p_r,z(t_r)) - \psi_r(p_r,y(t_r))\| \\
&&~~~~~~+ (t-p_r)^{1-\eta} \Bigg{(}\int_{p_r}^{t} (t-s)^{\eta - 1} \big{\|} T_\eta (t-s) \big{(}Bu(s) - Bv(s)\big{)} \big{\|}\\
&&~~~~~~ + \int_{p_r}^{t} (t-s)^{\eta - 1} \big{\|} T_\eta (t-s) \big{(} h(s,z(s))- h(s,y(s)) \big{)} \big{\|} \\
&&~~~~~~  + \int_{p_r}^{t} (t-s)^{\eta - 1} \big{\|} T_\eta (t-s) \big{(} \phi_r(s,z(s))-\phi_r(s,y(s)) \big{)} \big{\|} ds \Bigg{)}\\
&&~~ \leq \dfrac{M}{\Gamma(\eta)} \Bigg{[} b_r \|z(t_r) - y(t_r)\| + (t-p_r)^{1-\eta} \bigg{(}\int_{p_r}^{t} (t-s)^{\eta - 1} \|Bu(s) - Bv(s)\| ds\\
&&~~~~~~+ \int_{p_r}^{t} (t-s)^{\eta - 1} \tilde{\kappa}  {(s-p_r)^{1-\eta}} \|z(s)-y(s)\| ds \\
&&~~~~~~+ \int_{p_r}^{t} (t-s)^{\eta - 1} \| \phi_r(s,z(s))-\phi_r(s,y(s))\| \bigg{)} ds \Bigg{]}\\
&&~~\leq \dfrac{M}{\Gamma(\eta)} \Bigg{[} \dfrac{ b_r \|z-y\|_{r-1}}{(t_r - p_{r-1})^{1-\eta}} + (t-p_r)^{1 - \frac{1}{q}}{\bigg{(}\dfrac{q-1}{q \eta - 1 } \bigg{)}} ^ {\frac{q-1}{q}} \|Bu-Bv\|_{L^q} \\
&&~~~~~~+ \tilde{\kappa}(t-p_r)^{1-\eta} \int_{p_r}^{t} (t-s)^{\eta - 1} {(s-p_r)^{1-\eta}} \|z(s)-y(s)\| ds + \dfrac{\|z-y\|_{[0,a]}}{\eta \Gamma(1-\eta)} \Bigg{(} \sum_{k=0}^{r-1} {\Big{(}\dfrac{t_{k+1} - p_k}{p_r - t_{k+1}} \Big{)}}^\eta \\
&&~~~~~~+ \sum_{k=0}^{r-2} \dfrac{b_{k+1}{(t_{r+1} - p_r)}}{{(t_{k+1} - p_k)}^{1-\eta}{(p_r - p_{k+1})}^\eta}+ \dfrac{b_r (t_{r+1}-p_r)^{1-\eta}}{(t_{r}-p_{r-1})^{1-\eta}} \Gamma(1+\eta) \Gamma(1-\eta) \Bigg{)} \Bigg{]} \\
&&~~ \leq\dfrac{M}{\Gamma(\eta)} \Bigg{[}\dfrac{b_r\|z-y\|_{[0,a]}}{(t_r - p_{r-1})^{1-\eta}} + (t_{r+1}-p_r)^{1 - \frac{1}{q}}{\bigg{(}\dfrac{q-1}{q \eta - 1 } \bigg{)}} ^ {\frac{q-1}{q}} \|Bu-Bv\|_{L^q} \\
&&~~~~~~+ \tilde{\kappa} (t_{r+1}-p_r)^{1-\eta}\int_{p_r}^{t} (t-s)^{\eta - 1} {(s-p_r)^{1-\eta}} \|z(s)-y(s)\| ds + \dfrac{\|z-y\|_{[0,a]}}{\eta \Gamma(1-\eta)} \Bigg{(} \sum_{k=0}^{r-1} {\Big{(}\dfrac{t_{k+1} - p_k}{p_r - t_{k+1}} \Big{)}}^\eta \\
&&~~~~~~+ \sum_{k=0}^{r-2} \dfrac{b_{k+1}{(t_{r+1} - p_r)}}{{(t_{k+1} - p_k)}^{1-\eta}{(p_r - p_{k+1})}^\eta}+ \dfrac{b_r (t_{r+1}-p_r)^{1-\eta}}{(t_{r}-p_{r-1})^{1-\eta}} \Gamma(1+\eta) \Gamma(1-\eta) \Bigg{)} \Bigg{]} \\
&&~~\leq \mu \|z-y\|_{[0,a]} + \dfrac{M\tau^{1 - \frac{1}{q}}}{\Gamma(\eta)}{\bigg{(}\dfrac{q-1}{q \eta - 1 } \bigg{)}} ^ {\frac{q-1}{q}} \|Bu-Bv\|_{L^q}\\
&&~~~~~~ + \dfrac{M\tilde{\kappa}\tau^{1-\eta}}{\Gamma(\eta)} \int_{p_r}^{t} (t-s)^{\eta - 1} {(s-p_r)^{1-\eta}} \|z(s)-y(s)\| ds
\end{eqnarray*}
On applying generalised Gronwall's identity (\cite{zhou}), it leads to 
\begin{eqnarray}
(t-p_r)^{1-\eta} \|z(t) - y(t)\| \leq \left( \mu \|z-y\|_{[0,a]} + \varrho \|Bu-Bv\|_{L^q} \right) E_{\eta}(M\tilde{\kappa}\tau) \label{lemmacase1}
\end{eqnarray}
\textbf{Case 2:} For $t \in (t_r,p_r], r=1,2,\ldots,m$
\begin{eqnarray}
&&(t-p_{r-1})^{1-\eta} \|z(t) - y(t)\| \nonumber \\ 
&&\quad = (t-p_{r-1})^{1-\eta} \|\psi_r(t,z(t_{r}^{-}) - \psi_r(t,y(t_{r}^{-})\| \nonumber \\
&&\quad \leq b_r (t-p_{r-1})^{1-\eta} \| z(t_r) - y(t_r) \|_Z \nonumber \\
&&\quad \leq  \left( \mu \|z-y\|_{[0,a]} + \varrho \|Bu-Bv\|_{L^q} \right) E_{\eta}(M\tilde{\kappa}\tau) \label{lemmacase2}
\end{eqnarray}
On combining the results from Case 1 and Case 2, i.e. from eqns. (\ref{lemmacase1}) and (\ref{lemmacase2}), the following is obtained  
\begin{eqnarray*}
\|z-y\|_{[0,a]} \leq \left( \mu \|z-y\|_{[0,a]} + \varrho \|Bu-Bv\|_{L^q} \right) E_{\eta}(M\tilde{\kappa} \tau)
\end{eqnarray*}
and this leads to 
\begin{eqnarray}
\|z-y\|_{[0,a]} \leq \dfrac{\varrho E_\eta (M\tilde{\kappa}\tau)}{1-\mu E_\eta (M\tilde{\kappa} \tau)} \|Bu-Bv\|_{L^q}
\end{eqnarray}
This accomplishes the proof.
\end{proof}

\begin{theorem}
The nonlinear control system (\ref{system}) becomes approximately controllable, provided the assumptions $(H0), (H1)$ and $(H3)-(H5)$ holds true and A generates the differentiable semigroup $T(t)$.
\end{theorem}
\begin{proof}
In order to manifest approximate controllability of nonlinear control system (\ref{system}), 
it is adequate to claim that $D(A)\subset \overline{K_a(g)}$, as it is well known that domain of A, $D(A)$ is dense in Z. \\
\indent First, we will prove the approximate controllability of (\ref{system}) in $[0,t]$ for $t \in (0,t_1]$. For any $z_0 \in Z$, it is understood that ${t_1}^{\eta-1} T_\eta(t_1)z_0 \in D(A)$ because $T(t)$ is differential semigroup (implies $D(A)=Z$). Now, for $\wp \in D(A)$, existence of  a function $\zeta \in L^q([0,t_1];Z)$ can be shown such that 
$\mathbb{F}\zeta = \wp - {t_1}^{\eta-1} T_\eta(t_1)z_0$, like \\
$\zeta(t) = \dfrac{(\Gamma(\eta))^2 (t_1-t)^{1-\eta}}{t_1} \bigg{[} T_\eta(t_1-t)\\
 ~~~~~~~~~+2t  \dfrac{d T_\eta(t_1-t)}{dt} \bigg{]} \left( \wp - {t_1}^{\eta-1} T_\eta(t_1)z_0 \right),~t\in(0,t_1)$~ (see \cite{liu:siam}).\\
Next step is to show the existence of a control function $u_{\epsilon}^0(\cdot) \in L^q([0,t_1];U)$ in a way that the underneath inequality holds \\
$$ \| \wp-{t_1}^{\eta-1} T_\eta(t_1)z_0 - \mathbb{F} \Omega_h(z_\epsilon) - \mathbb{F}Bu_{\epsilon}^0 \|_Z < \epsilon$$\\
where $z_\epsilon(t)$ is a mild solution of the system(\ref{system}) in accord with the control $u_\epsilon^0(t)$.
With the use of hypothesis (H8), we can say that for any given $\epsilon >0$ and $u_{1}^{0}(\cdot) \in L^q([0,t_1];U)$, there exists a control $u_{2}^{0} (\cdot) \in L^q([0,t_1];U)$ satisfying \\
$$\| \wp-{t_1}^{\eta-1} T_\eta(t_1)z_0 - \mathbb{F} \Omega_h(z_1) - \mathbb{F}Bu_2^0 \|_Z < \dfrac{\epsilon}{2^2}$$ where $z_1(t) = z(t;u_1^0),~ t\in [0,t_1]$. Denote $z_2(t) = z(t,u_2^0)$, again by hypothesis (H8), there exists $\omega_2^0 \in L^q([0,t_1];U)$ satisfying
\begin{eqnarray*}
\| \mathbb{F}[\Omega_h(z_2)-\Omega_h(z_1)]- \mathbb{F}(B\omega_2^0)\|_{Z} < \frac{\epsilon}{2^3}
\end{eqnarray*}
and
\begin{eqnarray*}
\|B\omega_2^0\|_{L^q} &\leq & \aleph \| \Omega_h(z_2)-\Omega_h(z_1)\|_{L^q}\\
& = & \aleph \left( \int_{0}^{t_{1}} \| h(t,z_2(t)) - h(t,z_1(t)) \|_Z^q dt \right)^{\frac{1}{q}}\\
&\leq & \aleph \tilde{\kappa} \left( \int_{0}^{t_{1}} \left( t^{1-\eta}\| z_2(t) - z_1(t) \|_Z \right) ^q dt \right)^{\frac{1}{q}}\\
&\leq & \aleph \tilde{\kappa} t_1^{\frac{1}{q}} \|z_2 - z_1 \|_0 \\
&\leq & \aleph \tilde{\kappa} t_1^{\frac{1}{q}} \dfrac{\varrho E_\eta (M\tilde{\kappa}\tau)}{1-\mu E_\eta (M\tilde{\kappa}\tau)} \|Bu_2^0-Bu_1^0\|_{L^q}
\end{eqnarray*}
Now, define
\begin{eqnarray*}
u_3^0(t) = u_2^0(t) - \omega_2^0(t), ~~~u_3^0(t) \in U,
\end{eqnarray*}
then 
\begin{eqnarray*}
&&\| \wp-{t_1}^{\eta-1} T_\eta(t_1)z_0 - \mathbb{F} \Omega_h(z_2) - \mathbb{F}Bu_3^0 \|_Z \\
&& \quad \leq \| \wp-{t_1}^{\eta-1} T_\eta(t_1)z_0 - \mathbb{F} \Omega_h(z_1) - \mathbb{F}Bu_2^0 \|_Z + \|\mathbb{F}B\omega_2^0 - [\mathbb{F}\Omega_g(z_2)- \mathbb{F}\Omega_g(z_1)]\|_{Z}\\
&& \quad < \bigg{(} \frac{1}{2^2} + \frac{1}{2^3} \bigg{)} \epsilon .
\end{eqnarray*}
By applying inductions, a sequence $\{u_n^0 \}$ in $L^q([0,t_1];U)$ is obtained such that 
\begin{eqnarray*}
\| \wp-{t_1}^{\eta-1} T_\eta(t_1)z_0 - \mathbb{F}(z_n) - \mathbb{F}Bu_{n+1}^0\|_{Z} < \bigg{(} \frac{1}{2^2} + \frac{1}{2^3} + ... +\frac{1}{2^{n+1}} \bigg{)} \epsilon,
\end{eqnarray*}
where $z_n(t) = z(t,u_n^0(t))$ and   
\begin{eqnarray*}
&&\| Bu_{n+1}^0 - Bu_{n}^0\|_{L^q} <  \aleph \tilde{\kappa} t_1^{\frac{1}{q}} \dfrac{\varrho E_\eta (M\tilde{\kappa}\tau)}{1-\mu E_\eta (M\tilde{\kappa}\tau)} \|Bu_{n}^0 - Bu_{n-1}^0\|_{L^q}
\end{eqnarray*}
By (\ref{mainass}), it is evident that the sequence $\{Bu_n^0\}_{n \in \mathbb{N}}$ is a cauchy sequence in $L^q([0,t_1];Z)$. Thus, for any $\epsilon > 0$, a positive integer $n_0$ can be found satisfying 
$$\| \mathbb{F}Bu_{n_0 + 1}^0 - \mathbb{F}Bu_{n_0}^0\|_Z < \frac{\epsilon}{2}.$$
Therefore, 
\begin{eqnarray*}
&&\| \wp-{t_1}^{\eta-1} T_\eta(t_1)z_0 - \mathbb{F}\Omega_h(z_{n_0}) - \mathbb{F}Bu_{n_0}^{0}\|_Z \\
&& \quad \leq \| \wp-{t_1}^{\eta-1} T_\eta(t_1)z_0 - \mathbb{F}\Omega_h(z_{n_0}) - \mathbb{F}Bu_{n_0 +1}^0 \|_Z + \| \mathbb{F}Bu_{n_0 + 1}^0 - \mathbb{F}Bu_{n_0}^0\|_Z\\
&& \quad < \left( \frac{1}{2^2} + \frac{1}{2^3} + \cdots + \frac{1}{2^{n_0 + 1}} \right)\epsilon + \frac{\epsilon}{2} < \epsilon.
\end{eqnarray*}
Thus, the approximate controllability of (\ref{system}) is proved in the interval $[0,t_1]$.
\indent Further, we need to prove approximate controllability in $[0,t]$ for $t \in (p_1,t_2]$.\\
For any $\psi_1(p_1,z(t_1)) \in Z$, it is understood that ${(t_2-p_1)}^{\eta-1} T_\eta(t_2 - p_1)\psi_1(p_1,z(t_1)) \in D(A)$ because $T(t)$ is differential semigroup. Now, for $\wp \in D(A)$, existence of  a function $\zeta_1 \in L^q([p_1,t_2];Z)$ can be shown such that 
$\mathbb{F}\zeta_1 = \wp - {(t_2-p_1)}^{\eta-1} T_\eta(t_2 - p_1)\psi_1(p_1,z(t_1))$. \\
\indent Next step is to show the existence of a control function $u_\epsilon^1(\cdot) \in L^q([p_1,t_2];U)$ in a way that the underneath inequality holds \\
$$ \| \wp - {(t_2-p_1)}^{\eta-1} T_\eta(t_2 - p_1)\psi_1(p_1,z(t_1)) - \mathbb{F} \Omega_h(z_\epsilon) - \mathbb{F} \Omega_{\phi_1}(z_\epsilon)- \mathbb{F}Bu_\epsilon^1 \|_Z < \epsilon$$\\
where $z_\epsilon(t)$ is a mild solution of the system(\ref{system}) in accord with the control $u_\epsilon^1(t)$ and 
$$\wp - {(t_2-p_1)}^{\eta-1} T_\eta(t_2 - p_1)\psi_1(p_1,z(t_1))= \wp^* \in D(A).$$
With the use of hypothesis (H8), we can say that for any given $\epsilon >0$ and $u_{1}^{1}(\cdot) \in L^q([p_1,t_2];U)$, there exists a control $u_{2}^{1} (\cdot) \in L^q([p_1,t_2];U)$ satisfying \\
$$\| \wp^* - \mathbb{F} \Omega_h(z_1) - \mathbb{F} \Omega_{\phi_1}(z_1) -  \mathbb{F}Bu_2^1 \|_Z < \dfrac{\epsilon}{2^2}$$ where $z_1(t) = z(t;u_1^1),~ t\in (p_1,t_2]$. Denote $z_2(t) = z(t,u_2^1)$, again by hypothesis (H8), there exists $\omega_2^1 \in L^q([p_1,t_2];U)$ satisfying
\begin{eqnarray*}
\| \mathbb{F}[\Omega_h(z_2)-\Omega_h(z_1)]+ \mathbb{F} [\Omega_{\phi_1}(z_2)- \Omega_{\phi_2}(z_1)] - \mathbb{F}(B\omega_2^1)\|_{Z} < \frac{\epsilon}{2^3}
\end{eqnarray*}
and
\begin{eqnarray*}
&&\|B\omega_2^1\|_{L^q} \\
&& \quad \leq \aleph \left[ \| \Omega_h(z_2)-\Omega_h(z_1)\|_{L^q} + \| \Omega_{\phi_1}(z_2) - \Omega_{\phi_1}(z_1)\| \right] \\
&& \quad =  \aleph \left[ \left( \int_{0}^{t_{1}} \| h(t,z_2(t)) - h(t,z_1(t)) \|_Z^q dt \right)^{\frac{1}{q}} + \left( \int_{p_1}^{t_{2}} \| \phi_1(t,z_2(t)) - \phi_1(t,z_1(t)) \|_Z^q dt \right)^{\frac{1}{q}} \right] \\
&& \quad \leq \aleph \Bigg[ \tilde{\kappa} (t_2 - p_1)^{\frac{1}{q}}\|z_2 - z_1\|_1
 + \dfrac{1}{\Gamma(1-\eta)}\Bigg( \dfrac{(t_1-p_0)^\eta}{(p_1-t_1)^{1+\eta}} (t_2-p_1)^{\frac{1}{q}}\\
&& \quad \quad \quad + \dfrac{b_1}{(t_1-p_0)^{1-\eta}} \left(\int_{p_1}^{t_2}(t-p_1)^{-\eta q} dt \right)^{1/q} \Bigg) \|z_2 - z_1\|_0 \Bigg]\\
&&  \quad \leq \aleph \Bigg[ \tilde{\kappa} (t_2 - p_1)^{\frac{1}{q}}+ \dfrac{1}{\Gamma(1-\eta)}\Bigg( \dfrac{(t_1-p_0)^\eta}{(p_1-t_1)^{1+\eta}} (t_2-p_1)^{\frac{1}{q}}\\
&& \quad \quad \quad + \dfrac{b_1 (t_2-p_1)^{\frac{1}{q}-\eta}}{(t_1-p_0)^{1-\eta} (1-\eta q)^{\frac{1}{q}}} \Bigg) \Bigg]\dfrac{\varrho E_\eta (M\tilde{\kappa}\tau)}{1-\mu E_\eta (M\tilde{\kappa}\tau)} \|Bu_{2}^1 - Bu_{1}^1\|_{L^q}
\end{eqnarray*}
Now, define
\begin{eqnarray*}
u_3^1(t) = u_2^1(t) - \omega_2^1(t), ~~~u_3^1(t) \in U,
\end{eqnarray*}
then 
\begin{eqnarray*}
&&\| \wp^* - \mathbb{F} \Omega_h(z_2) - \mathbb{F} \Omega_{\phi_1}(z_2)-\mathbb{F}Bu_3^1 \|_Z \\
&& \quad \leq \| \wp^* - \mathbb{F} \Omega_h(z_1) - \mathbb{F} \Omega_{\phi_1}(z_1)-\mathbb{F}Bu_2^1 \|_Z \\
&& \quad \quad + \| \mathbb{F}B\omega_2^1 - [\mathbb{F}\Omega_h(z_2)- \mathbb{F}\Omega_h(z_1)] - [\mathbb{F}\Omega_{\phi_1}(z_2)- \mathbb{F}\Omega_{\phi_1}(z_1)]\|_{Z}\\
&& \quad < \bigg{(} \frac{1}{2^2} + \frac{1}{2^3} \bigg{)} \epsilon .
\end{eqnarray*}
By applying inductions, a sequence $\{u_n^1 \}$ in $L^q([p_1,t_2];U)$ is obtained such that 
\begin{eqnarray*}
\| \wp^* - \mathbb{F}\Omega_h(z_n) - \mathbb{F}\Omega_{\phi_1}(z_n)- \mathbb{F}Bu_{n+1}^0\|_{Z} < \bigg{(} \frac{1}{2^2} + \frac{1}{2^3} + \cdots +\frac{1}{2^{n+1}} \bigg{)} \epsilon,
\end{eqnarray*}
where $z_n(t) = z(t,u_n^1(t))$ and   
\begin{eqnarray*}
&&\| Bu_{n+1}^1 - Bu_{n}^1\|_{L^q} \\
&& \quad < \aleph \Bigg[ \tilde{\kappa} (t_2 - p_1)^{\frac{1}{q}}+ \dfrac{1}{\Gamma(1-\eta)}\bigg( \dfrac{(t_1-p_0)^\eta}{(p_1-t_1)^{1+\eta}} (t_2-p_1)^{\frac{1}{q}}\\
&& \quad \quad \quad + \dfrac{b_1 (t_2-p_1)^{\frac{1}{q}-\eta}}{(t_1-p_0)^{1-\eta} (1-\eta q)^{\frac{1}{q}}} \bigg) \Bigg]\dfrac{\varrho E_\eta (M\tilde{\kappa}\tau)}{1-\mu E_\eta (M\tilde{\kappa}\tau)} \|Bu_{n}^1 - Bu_{n-1}^1\|_{L^q}
\end{eqnarray*}
By (\ref{mainass}), it is evident that the sequence $\{Bu_n^1\}_{n \in \mathbb{N}}$ is a cauchy sequence on $L^q([p_1,t_2];Z)$. Thus, for any $\epsilon > 0$, a positive integer $n_0$ can be found satisfying 
$$\| \mathbb{F}Bu_{n_0 + 1}^1 - \mathbb{F}Bu_{n_0}^1\|_Z < \frac{\epsilon}{2}.$$
Therefore, 
\begin{eqnarray*}
&&\| \wp^* - \mathbb{F}\Omega_h(z_{n_0})-\mathbb{F}\Omega_{\phi_1}(z_{n_0})- \mathbb{F}Bu_{n_0}^1\|_Z \\
&& \quad \leq \| \wp^* - \mathbb{F}\Omega_h(z_{n_0}) -\mathbb{F}\Omega_{\phi_1}(z_{n_0})- \mathbb{F}Bu_{n_0 +1}^1 \|_Z + \| \mathbb{F}Bu_{n_0 + 1}^1 - \mathbb{F}Bu_{n_0}^1\|_Z\\
&& \quad < \left( \frac{1}{2^2} + \frac{1}{2^3} + \cdots + \frac{1}{2^{n_0 + 1}} \right)\epsilon + \frac{\epsilon}{2} < \epsilon.
\end{eqnarray*}
Thus, the approximate controllability of (\ref{system}) is proved in the interval $[0,t_2]$.
Similarly, repeating the process for $r=2,3,\ldots,m$, we finally get 
$$ \| \wp - {(a-p_m)}^{\eta-1} T_\eta(a - p_m)\psi_1(p_m,z(t_m)) - \mathbb{F} \Omega_h(z_\epsilon) - \mathbb{F} \Omega_{\phi_m}(z_\epsilon)- \mathbb{F}Bu_\epsilon^m \|_Z < \epsilon$$ 
which establishes the approximate controllability of the system (\ref{system}) in $[0,a]$.
\begin{remark}
It should be noted that if the semigroup $T(t)$ is not differentiable, the proposed concepts in the above theorem holds valid for $z_0, \psi_r(p_r,z(t_r)) \in D(A)$ instead of $z_0, \psi_r(p_r,z(t_r)) \in Z$.\\
\end{remark}

\end{proof}

\section{Examples} \label{6}
\begin{example}
Examine the following fractional differential problem with Riemann-Liouville derivative involving non-instantaneous impulses:
\begin{eqnarray}
	&&_{0}D_{t}^{\frac{2}{3}}z(t,n) = A z(t,n) + Bu(t,n) + h(t,z(t,n)), ~t \in \cup_{r=0}^{m}(p_r,t_{r+1}] \subset [0,1],~n \in \mathbb{N},  \nonumber \\
&&z(t,n) = \psi_r(t,z(t_r,n)) , ~~ t \in \cup_{r=1}^{m}(t_r,p_r], ~n\in \mathbb{N}, \label{sysex1} \\
&&_{0}I_{t}^{\frac{1}{3}}z(t,x)|_{t=0} = z_{0,n},~~n \in \mathbb{N} \nonumber\\
&&_{p_r}I_t^{\frac{1}{3}}z(t,n)|_{t=p_r} = \psi_r(p_r, z(t_r,n)), ~~ n \in \mathbb{N}. \nonumber
\end{eqnarray}

The system (\ref{sysex1}) can be written as system (\ref{system}) in abstract form and thus follows approximate controllability from Theorem 2 under assumptions $H(1)$-$H(8)$.
\end{example}

\begin{remark}
The example presented above supports the methodology and points out the correction in example given in article \cite{liu:siam}.
\end{remark}

\begin{example}
Examine the below mentioned initial value problem with Riemann-Liouville derivative involving non-instantaneous impulses:
\begin{eqnarray}
	&&_{0}D_{t}^{\frac{2}{3}}z(t,x) = \frac{\partial^2}{\partial x^2} z(t,x) + u(t) + h(t,z(t,x)), ~t \in \cup_{r=0}^{m}(p_r,t_{r+1}] \subset [0,1],~x \in [0,\pi],  \nonumber \\
&&z(t,x) = \psi_r(t,z(t_r^-,x)) , ~~ t \in \cup_{r=1}^{m}(t_r,p_r], ~x\in [0,\pi], \label{sysex2} \\
&&z(t,0) = 0 = z(t,\pi), ~ t \in (0,1] \nonumber \\
&&_{0}I_{t}^{\frac{1}{3}}z(t,x)|_{t=0} = z_0(x) \in D(A) ,~~x \in [0,\pi] \nonumber\\
&&_{p_r}I_t^{\frac{1}{3}}z(t,x)|_{t=p_r} = \psi_r(p_r, z(t_r,x)) \in D(A), ~~ x \in [0,\pi] \nonumber
\end{eqnarray}
Let $W = W^{'} = L^2([0,\pi])$, the map $B=I$ and the operator $A:D(A) \subset W \rightarrow W$ defined as 
\begin{eqnarray*}
Aw = w''
\end{eqnarray*}
where
\begin{eqnarray*}
D(A) =  &&\Bigg \{ w\in W ~\mid~ w, ~\frac{\partial w}{\partial x} \text{ ~are absolutely continuous},~ \frac{\partial^2 w}{\partial x^2} \in W\\
&&\quad \text{ and } w(0) = 0 = w(\pi) \Bigg \}
\end{eqnarray*}
Then, A can be written in the form of  
\begin{eqnarray*}
Aw=  \sum_{l=1}^{\infty} (-l^2) \langle w, \alpha_l \rangle \alpha_l, ~~w \in D(A)
\end{eqnarray*}
where $\alpha_l(x)= \sqrt{\frac{2}{\pi}}sin~ lx~ (l \in \mathbb{N})$ are the eigen functions corresponding to the eigen values $-l^2$ respectively and $\{ \alpha_1, \alpha_2,\ldots\} $ forms an orthonormal basis of W. A semigroup $T(t) (t>0)$ in W having A as its infinitesimal generator is written as 
\begin{eqnarray*}
T(t)w = \sum_{l=1}^{\infty} e^{-l^2t} \langle w, \alpha_l \rangle \alpha_l, ~~w \in W~~ \text{and}~~ \|T(t)\| \leq e^{-1},~ M = 1.
\end{eqnarray*}
Let us choose the nonlinear function $h$ as \\
$h(t,z(t,x))= 1+(t-p_r)^2 + \delta (t-p_r)^\beta [ z(t,x) + sin z(t,x)]$,\\ where $t\in (p_r,p_{r+1}]$ and $r=0,1,\ldots,m$. $\delta$ and $\beta$ are constants and $\beta \geq 1-\eta$.\\
Now, 
\begin{eqnarray*}
\|h(t,z(t,x)) - h(t,y(t,x))\| &\leq & |\delta| (t-p_r)^\beta \| z(t,x)-y(t,x) + sin z(t,x) - sin y(t,x)\|\\
&\leq & |\delta| (t-p_r)^{\beta + \eta -1}(t-p_r)^{1 - \eta}\big[\|z(t,x)-y(t,x)\|\\
 & \quad + & \|2cos(\dfrac{z(t,x)+y(t,x)}{2})sin(\dfrac{z(t,x)-y(t,x)}{2})\| \big]\\
 & \leq & 2|\delta |(t-p_r)^{1 - \eta} \| z(t,x)-y(t,x) \|\\
 & \leq & 2|\delta| \| z(t,x)-y(t,x) \|
\end{eqnarray*}
and 
\begin{eqnarray*}
\|h(t,z(t,x))\| &\leq & 1+(t-p_r)^2 + |\delta | (t-p_r)^\beta \| sin z(t,x) + z(t,x) \|\\
& \leq & 1+(t-p_r)^2 + 2|\delta | (t-p_r)^{\beta + \eta -1}(t-p_r)^{1 - \eta}  \|z(t,x)\|\\
& \leq & (1+(t-p_r)^2) + 2|\delta | (t-p_r)^{1 - \eta} \|z(t,x)\|
\end{eqnarray*}
It is evident that assumptions $(H1),(H2)$ and $(H5)$ are satisfied with $\kappa = \tilde{\kappa} = d = 2|\delta |$. Similarly, the assumptions $(H3)$ and $(H6)$ are satisfied by choosing suitable impulsive functions. 
Further, the assumptions $(H4),(H7)$ and $(H8)$ are satisfied by choosing $\delta$ to be sufficiently close to zero. Thus, approximate controllability of (\ref{sysex2}) follows from Theorem $2$ without assuming $T(t)$ as differentiable semigroup.
%
\end{example}

\section{Concluding Remarks} \label{7}
The article established the results for existence and approximate controllability of the non-instantaneous impulsive fractional differential systems involving Riemann-Liouville derivatives. The definition of the mild solution of the system (\ref{1}) has been constructed utilizing the concept of fractional semigroup and appropriate integral type initial conditions. The existence and uniqueness results have been derived with the help of Banach's fixed point approach. Approximate controllability has been established using Lemma 3 along with iterative techniques for control sequence.
The study of non-instantaneous impulsive systems covers a wide range of applications. 
The proposed future work 
 emerges from relaxing the Lipschitz continuity on the nonlinear operator. If the nonlinear operator $h$ is not Lipschitz, then even existence of solution is the matter of prime investigation. Also, the 
present findings can be extended for the partial approximate controllability or finite approximate controllability of the considered system  with general nonlocal conditions. For some idea, see \cite{abdul:partial,mahm2,mahm3,jde1}.
%

\end{document}